\documentclass[11pt, reqno, fleqn]{amsart} %\reqno = numero equations a droite

\usepackage{amsmath,amscd,amssymb,amsfonts,amsthm,courier,relsize,bm}
\usepackage{hyperref,enumitem,mathrsfs,slashed,mathtools, extarrows}
\setlist{leftmargin=10mm}
\usepackage[bottom]{footmisc}
\usepackage{xcolor}  %the next few lines set the colors of the hyperlinks, and remove the color boxes
\usepackage{geometry}
\usepackage[T1]{fontenc}
\usepackage{tikz}

\usepackage[bitstream-charter]{mathdesign} %In this one, replace \nu with \!\nu
\usepackage[T1]{fontenc}

\hypersetup{
    colorlinks,
    linkcolor={red!50!black},
    citecolor={blue!70!black},
    urlcolor={blue!80!black}
}

\newtheoremstyle{mytheoremstyle} % name
    {4mm}                    % Space above
    {4mm}                    % Space below
    {\itshape}                   % Body font
    {10mm}                           % Indent amount
    {\bfseries}                   % Theorem head font
    {. \, --- \,}                          % Punctuation after theorem head
    {0.5em}                       % Space after theorem head
    {}  % Theorem head spec (can be left empty, meaning ?normal?)

\theoremstyle{mytheoremstyle}
\newtheorem{theorem}{Theorem}[section]
\newtheorem*{theorem*}{Theorem}

\newtheorem{proposition}[theorem]{Proposition}

\newtheorem{lemma}[theorem]{Lemma}

\newtheoremstyle{mydefinitionstyle} % name
    {4mm}                    % Space above
    {4mm}                    % Space below
    {}                   % Body font
    {10mm}                           % Indent amount
    {\bfseries}                   % Theorem head font
    {.\, --- \,}                          % Punctuation after theorem head
    {0.5em}                       % Space after theorem head
    {}  % Theorem head spec (can be left empty, meaning ?normal?)

\theoremstyle{mydefinitionstyle}    
\newtheorem{definition}[theorem]{Definition}

\newtheoremstyle{example}% name of the style to be used
  {4mm}% measure of space to leave above the theorem. E.g.: 3pt
  {4mm}% measure of space to leave below the theorem. E.g.: 3pt
  {}% name of font to use in the body of the theorem
  {10mm}% measure of space to indent
  {\bfseries}% name of head font
  {. \, --- \,}% punctuation between head and body
  { }% space after theorem head; " " = normal interword space
  {\thmname{#1}\thmnumber{ #2}\textnormal{\thmnote{ (#3)}}}

\theoremstyle{example}

\newtheorem*{example*}{Example}

\newtheorem*{examples*}{Examples}

\newtheorem*{remark*}{Remark}
\newtheorem*{remarks*}{Remarks}

\usepackage{xpatch}
\xpatchcmd{\proof}{\hskip\labelsep}{\hskip6.5\labelsep}{}{}  %% change 5 here as you wish

\newcommand{\ignore}[1]{}

\newcommand{\ol}[1]{\overline{#1}}

\newcommand{\ti}[1]{\widetilde{#1}}

\renewcommand{\i}{{\mathrm{i}}}
\newcommand{\HP}{\text{HP}}

\def\ad{\ensuremath{\textnormal{ad}}}

\def\E{\ensuremath{\mathscr{E}}}

\def\R{\ensuremath{\mathcal{R}}}

\def\A{\ensuremath{\mathscr{A}}}
\def\B{\ensuremath{\mathscr{B}}}
\def\L{\ensuremath{\mathscr{L}}}

\def\Rc{\ensuremath{\mathcal{R}}}

\def\C{\ensuremath{\mathbb{C}}}
\def\R{\ensuremath{\mathbb{R}}}
\def\Z{\ensuremath{\mathbb{Z}}}

\def\D{\ensuremath{\mathbf{D}}}
\def\dd{\ensuremath{\mathbf{d}}}
\def\bB{\ensuremath{\overline{B}}}

\def\End{\ensuremath{\textnormal{End}}}
\def\Hom{\ensuremath{\textnormal{Hom}}}

\def\id{\ensuremath{\textnormal{id}}}

\def\Tr{\ensuremath{\textnormal{Tr}}}

\def\index{\ensuremath{\textnormal{Index}}}

\def\dim{\ensuremath{\textnormal{dim}}}
\def\KK{\ensuremath{\textnormal{KK}}}

\def\CC{\ensuremath{\textnormal{CC}}}
\def\tr{\ensuremath{\textnormal{tr}}}
\def\ch{\ensuremath{\textnormal{ch}}}
\def\Th{\ensuremath{\widehat{T}}}

\usetikzlibrary{arrows,chains,matrix,positioning,scopes}
\makeatletter
\tikzset{join/.code=\tikzset{after node path={%
\ifx\tikzchainprevious\pgfutil@empty\else(\tikzchainprevious)%
edge[every join]#1(\tikzchaincurrent)\fi}}}
\makeatother
\tikzset{>=stealth',every on chain/.append style={join},
         every join/.style={->}}
\tikzstyle{labeled}=[execute at begin node=$\scriptstyle,
   execute at end node=$]
\begin{document}
\newgeometry{bottom=2.5cm, top=2cm, hmargin=2.5cm}
\sloppy
\title{Algebra cochains, the bivariant JLO cocycle and the Mathai--Quillen form}

\author{Rudy Rodsphon}
\date{August 1, 2022}
\address{Washington University \\ 1 Brookings Dr \\ Department of Mathematics $\&$ Statistics, Cupples Hall \\ St Louis MO 63130, United States}
\email{rrudy@wustl.edu}

\maketitle

\vspace{-0.5cm}
\setlength{\parindent}{0mm}
\setlength{\mathindent}{15mm}

\begin{abstract}
This is a first investigation by the author of the similarity between Quillen's superconnection formalism, his constructions of (periodic) cyclic cocycles via algebra cochains on a bar construction, and Kasparov bimodules for KK-theory. In this article, we do so by deriving a slight extension of the Mathai--Quillen Thom form via a bivariant JLO cocycle. The main idea (which is in fact not really new) is that KK-cycles should be thought of as superconnection forms; these methods will be applied to other contexts elsewhere.    
\end{abstract}

\vspace{5mm}

\section*{\centering Introduction} \setcounter{section}{0}
 
Historically, cyclic (co)homology came from two directions. In one, Connes introduced cyclic cohomology in order to find an extension of the de Rham cohomology to the realm of noncommutative geometry, and thereby a suitable receptacle for the Chern character (cf. \cite{ConIHES, ConBook}). In the other, cyclic cohomology occurs, in the work of Tsygan \cite{Tsy} and subsequently Loday--Quillen \cite{LQ}, in relationship to Lie theory and notably the Lie algebra of matrices, and makes extensive of a resolution of the cyclic complex via the cyclic bicomplex (commonly called the Loday--Quillen--Tsygan complex).   \\

In a remarkable paper \cite{Qui88}, Quillen merges these two  approaches of cyclic theory and introduces a framework that enables the construction of most of the interesting cyclic cohomology cocycles through methods imported from Chern--Weil theory. The main ingredients are: (i) the bar construction equipped with its natural dg-coalgebra struture, whose cocommutator subspace identifies to the cyclic bicomplex (up to a dimension shift); (ii) the fact that this structure gives rise to a dg-algebra structure on cochains, in which the bar differential and 1-cochains play respectively fulfill the roles that the de Rham differential and connection 1-forms play in the context of connection and curvature calculations. As a byproduct, most of the cyclic classes relevant to index theory may be obtained by exponentiating the curvature of a superconnection in the algebra of bar cochains; K-homology/K-theory cycles are to be viewed as a superconnection forms. \\

Later on, Perrot \cite{Per02} extended the algebra cochain formalism to the bivariant setting, and produces bivariant cyclic classes starting from a Kasparov bimodule, by viewing the latter as a superconnection form. This leads in turn to the construction of a bivariant Chern character via an explicit formula. \\

This set of ideas have found interesting applications, among which we shall only mention the elegant construction of the Connes--Moscovici residue cocycle \cite{CM95} found by Higson in \cite{Hig06}, and more recently the equivariant index theorems for non-proper, non-isometric actions obtained by Perrot \cite{Per2014a, Per2014b}, which in turn have found applications to the transverse index problem of Connes--Moscovici \cite{CM95, CM98} in subsequent joint work with the author \cite{PerRod16}. \\

In this article, our objectives are more modest, and we simply recompile different aspects of the algebra cochain formalism to prove that the Mathai--Quillen Thom form can be recovered via these methods, promoting it to a bivariant cyclic cocycle. The idea is essentially to compute (a part of) the bivariant Chern character of the Thom isomorphism, by interpreting suitably the representative of its KK-class as a superconnection form. Actually, this result is already known and mentioned in the original paper of Mathai-- Quillen, though not exactly in this format. \\

Throughout the whole process, we shall reexamine, at least in the aforementioned example, that the algebra cochain formalism completely subsumes the classical superconnection formalism, in such a way that the similarity we can draw with the Chern--Weil theory is not merely an analogy; its mechanisms may be reproduced practically mutadis mutandis. \\

While trying to understand the thought process that led Quillen to such remarkable constructions via his mathematical journal, it was interesting to note that he already knew cyclic cohomology, KK-theory, the Kasparov product and its application to the families index theorem, before his work on superconnections. Knowing that the latter was motivated by the local families index theorem, it might be speculated that he introduced superconnections as a device to relate KK-theory to cohomology, the same way as is done in K-theory via connections and the Chern character. The idea of constructing a bivariant Chern character that is functorially compatible with KK-theory, and especially with the Kasparov product, was achieved much later in joint work with Cuntz \cite{CQ98}), but we think it could have been in his mind since the early stages of noncommutative geometry. This would call a historical investigation, that may appear in future works together with an extended version of this note, including other known examples of cocycles obtained via the algebra cochain formalism. \\    

Other applications of the techniques in this paper will also be treated elsewhere. Among these, we mention the reduction of classical equivariant index problems to the case of a free action\footnote{This idea is already well-known in the folklore...}. In another direction, These techniques may also be applied to connect Kasparov's recently developed approach to transverse index theory to the work of Berline--Vergne/Paradan--Vergne \cite{BV96, PV08}. More speculatively, we think another potential application could be a direct construction of the fundamental cyclic class in the context of the algebraic index theorem \cite{NT96, FFS04}. 

\vspace{3mm}
  
\subsection*{Plan of the article} Section 1 is a preliminary section featuring some material about Quillen's algebra cochain formalism, its extension to the bivariant situation, outlining the proofs of some results that will be useful to use further. Section 2 applies the theory to examples: first with the special case of the Bott element, and then showing how to recover the Mathai--Quillen form via the Thom element viewed as a Kasparov bimodule.   

\vspace{3mm}

\subsection*{Acknowledgements} The authors wants to thank Denis Perrot and Xiang Tang warmly for interesting discussions and encouragements. 

\vspace{3mm}

\section{Cyclic theory and algebra cochains}

This section is essentially a recapitulation of the elements of cyclic theory that we will need, and sets the notations that will be used throughout. It may also be of some use to readers who are not acquainted with Quillen's approach to cyclic (co)homology, and especially his algebra cochains formalism \cite{Qui88}, which is the basis in Perrot's approach of the bivariant JLO formula/Chern character \cite{Per02}. 

\subsection{Algebra of noncommutative differential forms and the $(b,B)$-bicomplex}

Let $\A$ be an associative algebra over $\C$, and let $\ti{A}=\C\oplus \A$ be its unitalization.  Recall that the \emph{algebra of noncommutative differential forms} $(\Omega \A, d)$ is the free algebra generated by $a \in \A$ and symbols $db$ with $b \in \A$, where $d: \A \to \A$ is a linear map satisfying the Leibniz rule 
\[ d(a_0 a_1) = a_0 \cdot da_1 + da_0 \cdot a_1 \]
$\Omega \A$ naturally comes equipped with a filtration, 
\[ \Omega \A = \bigoplus_{k \geq 0} \Omega^{k} \A \quad ; \quad \Omega^{k} \A = \{ a_0 da_1 \ldots da_k \, ; \, a_0, \ldots, a_k \in A \}.\]
Moreover, there is a natural isomorphism of vector spaces: 
\begin{align*}
&\Omega^{k} \A \overset{\displaystyle{\simeq}}{\longrightarrow} \ti{A} \otimes \A^{\otimes k} \quad ; &a_0 da_1 \ldots da_k  & \longmapsto  a_0 \otimes a_1 \otimes \ldots \otimes a_k &&&&&&&&&&&&&&&&&&&&&&&&&&&&\\
& & da_1 \ldots da_k & \longmapsto 1 \otimes a_1 \otimes \ldots \otimes a_k &&&&&&&&&&&&&&&&&&&&&&&&&&&&
\end{align*}
The universal differential $d$ is extended to $\Omega \A$ by setting
\begin{align*}
d (a_0 d a_1 \ldots d a_k) = d a_0 d a_1 \ldots d a_k \quad ; \quad d (d a_1 \ldots d a_k) = 0 
\end{align*}
Let $\natural: \Omega \A \to \Omega \A_{\natural} := \Omega \A/[\Omega \A, \Omega \A]$ be the canonical projection. The composite map $\natural d: \Omega \A \to \Omega \A_{\natural}$ vanishes on the commutator subspace $[\Omega \A, \Omega \A]$, hence it factors through a map $\natural d: \Omega \A_{\natural} \to \Omega \A_{\natural}$ whose square is zero. The cohomology of the corresponding $\Z_2$-graded complex 
\[ \begin{tikzpicture} 
\matrix(m)[matrix of math nodes, row sep=3em, column sep=2.5em, text height=2ex, text depth=0.25ex]{ \Omega^{\mathrm{even}} \A  & \Omega^{\mathrm{odd}} \A_{\natural} \\ }; 
\path[->,font=\scriptsize, >=angle 90] 
(m-1-1)([yshift= 2pt]m-1-1.east) edge node[above] {$\natural d$} ([yshift= 2pt]m-1-2.west) (m-1-2)
(m-1-2)([yshift= -3pt]m-1-2.west) edge node[below] {$\natural d$} ([yshift= -3pt]m-1-1.east) (m-1-1) ;
\end{tikzpicture} \]
is Karoubi's \emph{noncommutative de Rham homology} $H^{\bullet}_{\mathrm{dR}}(\A)$. \\

One then defines two differentials on $\Omega \A$: the \emph{Hochschild boundary} $b : \Omega^{k+1} \A \to \Omega^{k} \A$ and \emph{Connes' operator} $B : \Omega^k \A \to \Omega^{k+1} \A$ defined as follows: 
\begin{multline*}
b(a_0 d a_1 \ldots d a_{k+1}) = a_0 a_1 d a_2 \ldots d a_{k+1} + \sum_{i=1}^{k} (-1)^{i} a_0 d a_1 \ldots d (a_i a_{i+1}) \ldots d a_{k+1} \\
+ (-1)^{k+1} a_{k+1} a_0 d a_1 \ldots d a_k,
\end{multline*}
\vspace{-0.5mm}
\[ B(a_0 d a_1 \ldots d a_k) = d a_0 d a_1 \ldots d a_k + (-1)^{k} d a_n d a_0 d a_1 \ldots d a_{k-1} + \ldots + (-1)^{kk}  d a_1 \ldots d a_{k} d a_0. \]

\vspace{1mm}

As $b^2 = B^2 = Bb+bB = 0$, these differentials make $\Omega \A$ into a bicomplex called the \emph{$(b,B)$-bicomplex in homology}. \\

\ignore{\[\begin{tikzpicture}
\matrix(m)[matrix of math nodes, row sep=3em, column sep=2.5em, text height=2ex, text depth=0.25ex]
    {        & \vdots     & \vdots      & \vdots   \\
      \ldots & \Omega^2 \A & \Omega^1 \A  & \Omega^0 \A  \\ 
      \ldots & \Omega^1 \A & \Omega^0 \A  & \\ 
      \ldots & \Omega^0 \A &             & \\};
\path[->, font=\scriptsize, >=angle 90] 
(m-1-2) edge node[left]{$b$} (m-2-2) 
(m-2-2) edge node[left]{$b$} (m-3-2)
(m-3-2) edge node[left]{$b$} (m-4-2)
(m-1-3) edge node[left]{$b$} (m-2-3) 
(m-2-3) edge node[left]{$b$} (m-3-3)
(m-1-4) edge node[left]{$b$} (m-2-4)
(m-2-4) edge node[above]{$B$} (m-2-3) 
(m-2-3) edge node[above]{$B$} (m-2-2)
(m-2-2) edge node[above]{$B$} (m-2-1)
(m-3-3) edge node[above]{$B$} (m-3-2)
(m-3-2) edge node[above]{$B$} (m-3-1)
(m-4-2) edge node[above]{$B$} (m-4-1);
\end{tikzpicture} \]}

Let $\widehat{\Omega}\A$ denote the direct product $\prod_{k} \Omega^{k} \A$. The \emph{periodic cyclic homology} $\HP_{\bullet}(\A)$ of $\A$ is the homology of the 2-periodic complex
\[ \begin{tikzpicture} 
\matrix(m)[matrix of math nodes, row sep=3em, column sep=2.5em, text height=2ex, text depth=0.25ex]{ \widehat{\Omega}^{\mathrm{even}} \A  & \widehat{\Omega}^{\mathrm{odd}} \A \\ }; 
\path[->,font=\scriptsize, >=angle 90] 
(m-1-1)([yshift= 2pt]m-1-1.east) edge node[above] {$B+b$} ([yshift= 2pt]m-1-2.west) (m-1-2)
(m-1-2)([yshift= -3pt]m-1-2.west) edge node[below] {$B+b$} ([yshift= -3pt]m-1-1.east) (m-1-1) ;
\end{tikzpicture} \]
where 
\begin{align*}
&\widehat{\Omega}^{\mathrm{even}} \A = \prod_{k \geq 0} \Omega^{2k} \A, \quad \widehat{\Omega}^{\mathrm{odd}} \A = \prod_{k \geq 0} \Omega^{2k+1} \A
\end{align*}
Instead of a direct sum, the direct product is required to get a non-trivial homology. \\

For $k \geq 0$, let $\CC^{k}(\A)$ be the dual of $\Omega^k\A$, that is, the space of $(k+1)$-linear forms on $\ti{A}$ verifying $ \varphi(a_0, \ldots, a_k) = 0 $ if $a_i = 1$ for at least one $i \geq 1$. The 2-periodic complex leads by duality to \emph{$(b,B)$-bicomplex in cohomology}: the (continuous) dual of $\widehat{\Omega} \A$ (for the filtration topology) is the direct sum
\[ \CC^{\bullet}(\A) = \bigoplus_{k \geq 0} \CC^{k}(\A). \]
The \emph{periodic cyclic cohomology} $\HP^{\bullet}(\A)$ of $\A$ is the cohomology of the dual $2$-periodic complex giving cyclic homology, or equivalently, the total complex of the \emph{$(b,B)$-bicomplex in cohomology}
\[ \begin{tikzpicture} 
\node (a) at (0,0) {$\CC^{\mathrm{even}}(\A) $}; 
\node (b) at (2.7,0) {$\CC^{\mathrm{odd}}(\A) $}; 
\path[->,font=\scriptsize, >=angle 90] 
([yshift= 2pt]a.east) edge node[above] {$B+b$} ([yshift= 2pt]b.west) 
([yshift= -3pt]b.west) edge node[below] {$B+b$} ([yshift= -3pt]a.east);
\end{tikzpicture} \]
where
\begin{align*}
&\CC^\mathrm{even}(\A) = \bigoplus_{k \geq 0} \CC^{2k}(\A), \quad \CC^\mathrm{odd}(\A) = \bigoplus_{k \geq 0} \CC^{2k+1}(\A)
\end{align*}

 %The universal algebra $\Omega A$ serves to define the \emph{cyclic homology} of $A$. Let us recall the $(b,B)$-complex for homology, i.e its dual complex gives the $(b,B)$-complex for cohomology given in Introduction, Section \ref{section intro Bb complexe}. 

\vspace{1mm}

\subsection{Algebra cochains} \label{Construction Quillen}

Suppose that $\ti{A}=\C\oplus{A}$ is an augmented algebra. The \emph{bar construction} $\bB(\ti{\A})$ of $\ti{A}$, denoted simply $\bB$ when the context is clear, is the differential graded (dg) coalgebra $\bB(\A) = \bB = \bigoplus_{n \geq 0} \bB_n$, where $\bB_n = \A^{\otimes n}$ for $n \geq 0$, with \emph{coproduct} $\Delta : \bB \to \bB \otimes \bB$
\begin{equation*}\Delta (a_1, \ldots, a_n) = \sum_{i=0}^n (a_1, \ldots, a_i) \otimes (a_{i+1}, \ldots, a_n) \end{equation*}
 
The \emph{counit} map $\eta$ is the projection onto $\A^{\otimes 0} = \C$, and the \emph{differential} is $b'$ :
\begin{equation*} b'(a_1, \ldots, a_{n+1}) = \sum_{i=1}^n (-1)^{i-1}(a_1, \ldots, a_i a_{i+1}, \ldots, a_{n+1}) \end{equation*}
which is defined as the zero-map on $\bB_0$ and $\bB_1$. \\

These operations endow $\bB$ with a structure of dg-coalgebra. Note that $b'$ and $\eta$ are morphisms of (graded) complexes, i.e $\Delta b' = (b' \otimes \id + \id \otimes b') \Delta$ and $\eta b' = b' \eta$. In addition, we use a standard abuse of notation identifies strings in the terms $i=0$ and $i=n$ of the coproduct formula as $(a_1, \ldots, a_n) \otimes 1$ and $1 \otimes (a_1 \ldots, a_n)$. When $\A$ is unital, note here that $1 \in \A^{\otimes 0} = \C$ \textbf{is not} the unit of $\A$. \\

Let $\Hom(\bB,\L)$ denote the space of $n$-linear maps over $\A$ with values in a (dg-)algebra $\L$. It has a differential $b'$: 
\begin{equation*} \delta f = (-1)^{n+1} fb' \end{equation*}
for $f \in \Hom^n(\bB,\L)$. The coproduct on the bar construction induces a product on $\Hom(\bB,L)$: if f and g are respectively cochains of degrees $p$ and $q$,  it is given by
\begin{equation*} 
fg(a_1, \ldots, a_{p+q}) = (-1)^{pq} f(a_1, \ldots, a_p)g(a_{p+1}, \ldots, a_{p+q}) 
\end{equation*}
Therefore, $\Hom(\bB,\L)$ has a structure of dg-algebra. 

\begin{example*} \label{ex connection curvature} Let $\rho : \A \to \L$ be a 1-cochain, i.e a linear map. See $\rho$ as a "connection" and define its "curvature" $\omega = \delta\rho + \rho^{2}$. An easy calculation shows that 
\[ \omega(a_1, a_2) = \rho(a_1 a_2) - \rho(a_1) \rho(a_2) \]
that is, $\rho$ is a homomorphism of algebras if and only if its curvature vanishes. One has a Bianchi-type identity 
\[ \delta \omega = - [\rho, \omega] \]
and more generally, using that $\delta$ and $[\rho, \centerdot]$ are derivations, we get by induction
\[ \delta \omega^{n} = - [\rho, \omega^{n}]. \hfill{\square} \]
\end{example*}

%\vspace{2mm}

We next define $\Omega_1 \bB$ and $\Omega_1^{\natural} \bB$ to be the following bicomodules over $\bB = \bB$:
\begin{align*}
& \Omega_1 \bB := \bB \otimes \ti{\A} \otimes \bB, \quad  \Omega_1^{\natural} \bB := \ti{\A} \otimes \bB = \Omega \A
\end{align*}
Notice that $\Omega_1^{\natural}(\bB)$ is the algebra of noncommutative differential forms $\Omega\A$. The exponent $\natural$ means that $\Omega_1^{\natural}\bB$ is the cocommutator subspace of $\Omega_1 \bB$, but this won't be used in the sequel. \\

Therefore, we have three different types of bar cochains, that we distinguish through the following terminology: \emph{$\Omega$-cochains} are elements of $\Hom(\Omega_1 \bB, \L)$, and \emph{Hochschild cochains} are elements of $\Hom(\Omega \A, \L)$. \emph{Bar cochains} is then be used for elements of $\Hom(\bB, \L)$. In general, the \emph{degree} of a cochain is the \emph{total degree}, meaning for example that if $f \in \Hom(\bB_p, \L_q)$, then the degree of $f$ is $|f|=p+q$; $p$ and $q$ are respectively referred to as the $\A$- and $\L$- degrees. This terminology applies to other types of cochains as well, keeping in mind that elements in $\A$ have degree $1$.\\

The $\bB$-bicomodule structure of $\Omega_1 \bB$ naturally induces a $\Hom(\bB, \L)$-bimodule structure on the space $\Hom(\Omega_1 \bB, \L)$ of $\Omega$-cochains, which is precisely described via the operation:
\begin{align*}
(\gamma \cdot f)(a_1, \ldots, a_{p-1}) & \otimes a_p \otimes  (a_{p+1}, \ldots, a_n)&\, \\
& = \, m(\gamma \otimes f)(1 \otimes \Delta_{\text{right}})(a_1, \ldots, a_{p-1}) \otimes a_p \otimes (a_{p+1}, \ldots, a_n)&  \\
&\, = \, \sum_{i=0}^{n-p} (-1)^{i|\gamma|}\gamma\big((a_1, \ldots, a_{p-1}) \otimes a_p \otimes (a_{p+1}, \ldots, a_{p+i})\big) f(a_{p+i+1}, \ldots, a_n).& 
\end{align*}
for the right-module structure; for the left-module structure, $f \cdot \gamma$ may be described by a formula of the same type. The symbol $m$ denotes multiplication in $\L$. \\

The differential $b'$ on $\bB$ naturally induces a differential $b''$ on $\Omega_1 \bB$, so the differential $\delta$ on bar cochains naturally carries over $\Omega$-cochains: for an $\Omega$-cochain $\gamma \in \Hom(\Omega_1^{\natural} \bB, \L)$, one defines
\begin{gather*}
\delta \gamma = (-1)^{|\gamma|+1} \gamma b'' \, , \: \: \: \text{where} \\
\begin{split}
b''(a_1, \ldots, a_{p-1}) \otimes a_p \otimes (a_{p+1}, \ldots, a_n) \,\, = &\,\,\,\, b'(a_1, \ldots, a_{p-1}) \otimes a_p \otimes (a_{p+1}, \ldots, a_n) \\ 
& + (-1)^{p-2} (a_1, \ldots, a_{p-2}) \otimes a_{p-1} a_p \otimes (a_{p+1}, \ldots, a_n) \\
& + (-1)^{p-1} (a_1, \ldots, a_{p-1}) \otimes a_p a_{p+1} \otimes (a_{p+2}, \ldots, a_n) \\
& + (-1)^{p} (a_1, \ldots, a_{p-1}) \otimes a_p \otimes b'(a_{p+1}, \ldots, a_n)
\end{split}
\end{gather*} 

With this at hand, we may now give Quillen's description of the $(b,B)$-bicomplex. 
%

%\setlength{\parindent}{6mm}
%\textsc{Important fact.} A cochain $f$ of this kind has three degrees : a $\A$-degree as a multilinear map over $\A$, a $L$ degree and a total degree $f$, which is sum. This is the one which will be considered. \\

%There is a canonical injection $\natural : \Omega \A \to \Omega_1 \bB$: 
%\begin{equation*}
%\natural (a_1 \otimes (a_2, \ldots, a_n)) = \sum_{i=1}^n (-1)^{i(n-1)} (a_{i+1}, \ldots, a_n) \otimes a_1 \otimes  (a_2, \ldots, a_i) 
%\end{equation*}
%Given a (graded) trace $\tau : \L \longrightarrow \C$, we then obtain a morphism of complexes
%\begin{align*}
%\begin{array}{c c c c}
%\tau^\natural : & \Hom(\Omega_1 \bB,\L) & \longrightarrow & \Hom(\Omega \A, \C) \\
%                &  f                 & \longmapsto     & \tau^\natural(f) = \tau f \natural
%\end{array}
%\end{align*} 

\begin{theorem} \label{thm: 2-periodic complex} One has a 2-periodic complex:
\[ \begin{tikzpicture} 
\matrix(m)[matrix of math nodes, row sep=3em, column sep=2.5em, text height=2ex, text depth=0.25ex]{ \bB=\bB(\ti{\A})  & \Omega \A \\ }; 
\path[->,font=\scriptsize, >=angle 90] 
(m-1-1)([yshift= 2pt]m-1-1.east) edge node[above] {$\beta$} ([yshift= 2pt]m-1-2.west) (m-1-2)
(m-1-2)([yshift= -3pt]m-1-2.west) edge node[below] {$\overline{\partial}$} ([yshift= -3pt]m-1-1.east) (m-1-1) ;
\end{tikzpicture} \]
%\begin{equation*}\xymatrix{
%\ldots \ar[r]^-{\overline{\partial}} & B \ar[r]^-{\beta}  & \Omega^{B, \natural} \ar[r]^-{\overline{\partial}} & B \ar[r]^-\beta & \ldots} \end{equation*}
where \; $\beta: \bB \to \Omega \A$ \; ; \; $\overline{\partial} = \partial \natural: \Omega \A \to \bB $ \; ; \; $\natural: \Omega \A \to \Omega_1\bB$ \; ; \; $\partial: \Omega_1 \bB \to \bB$ are morphisms of complexes defined as follows:
\begin{gather*}
\beta(a_1, \ldots, a_n) = (-1)^{n-1} a_n \otimes (a_1, \ldots, a_{n-1}) - a_1 \otimes (a_2, \ldots, a_n) \\
\natural (a_1 \otimes (a_2, \ldots, a_n)) = \sum_{i=1}^n (-1)^{i(n-1)} (a_{i+1}, \ldots, a_n) \otimes a_1 \otimes  (a_2, \ldots, a_i) \\
\partial (a_{1}, \ldots, a_{p-1}) \otimes a_p \otimes  (a_{p+1}, \ldots, a_n) = (a_1, \ldots, a_n) \\
\overline{\partial} (a_1 \otimes (a_2, \ldots, a_n)) = \sum_{i=1}^n (-1)^{i(n-1)} (a_{i+1}, \ldots, a_n, a_1, a_2, \ldots, a_i) 
\end{gather*}
\end{theorem}

\begin{remark*} All the operators above are subsequently defined on cochains by duality. Note also that $\partial$ being a coderivation, it induces at such a derivation on $\Omega$-cochains. 
\end{remark*}

A direct calculation shows that $b'' \natural = \natural b$, so the differential $\delta$ induced on $\Omega_1^{\natural} \bB = \Omega \A$ is the Hochschild boundary. Consequently, we deduce that the complex $(\Hom(\Omega_1^ \natural \bB, \C), b)$ is isomorphic to the Hochschild complex $(\text{CC}^{\bullet}(\A), b)$, with degrees shifted by one. It is then clear that the above 2-periodic complex exactly reconstitutes the Loday--Quillen cyclic bicomplex, as $\beta$ and $\ol{\partial}$ correspond to its (horizontal) arrows. Therefore, the 2-periodic complex above is equivalent to Connes' $(b,B)$-bicomplex. The result below provides a method to produce interesting $(b,B)$-cocycles via the formalism of algebra cochains.

%Now, let $\L$ be a differential graded algebra. The maps $\overline{\partial}$ and $\beta$ of the periodic complex induces maps from bar cochains to Hochschild cochains (with values in $L$) and conversely by pull-back. The following formula is a key step.

\begin{proposition} \label{thm: Quillen (b,B)-cocycles} Let $V$ be a vector space, let $\psi \in \Hom(\Omega \A,V)$, and let $\varphi \in \Hom(\bB,V)$ be the bar cochain defined by
\begin{equation*} 
\varphi_n(a_1, \ldots, a_n) = \psi_{n+1}(1, a_1, \ldots, a_n)
\end{equation*}   
Suppose that for each $n$, we have
\begin{align*}
\delta \psi_{n+1} = (-1)^n \overline{\partial} \varphi_{n+2}
\end{align*} %&\delta \varphi_n = (-1)^{n} \beta \psi_{n+1}, \quad  
and that in addition, $\psi_{n+1}(a_0,a_1, \ldots, a_n) = 0$ whenever $a_i = 1$ for at least one $i \geq 1$. Then, for all $n$, 
\[ B \psi_{n+1} = b \psi_{n-1}. \] 
i.e $(B-b)\psi = 0$.
\end{proposition}

Discussing the proof will be useful in the sequel, so we provide it.

\begin{proof}
Notice that 
\[B\psi_{n+1}=\psi_{n+1}(1,\ol{\partial}(\centerdot)) = \ol{\partial}\varphi_n = (-1)^{n-2} \delta \psi_{n-1} = b\psi_{n-1}. \]
\end{proof}

Finally, two useful remarks: 
\begin{itemize}
\item[$\bullet$] The result above admits an immediate generalization if $V$ is a complex with differential $d$: in this case, one replaces $\delta$ by $\delta + d$, and the final conclusion is that $(B-b) \psi = \psi (B-b) = d\psi$. 
\item[$\bullet$] Notice that when the dg-algebra $\L$ comes equipped with a (graded) trace $\tau: \L \to V$, where $V$ is a vector space; we have a morphism of complexes 
\begin{align*}
\begin{array}{c c c c}
\tau^\natural : & \Hom(\Omega_1 \bB, \L) & \longrightarrow & \Hom(\Omega \A,V) \\
                &  f                 & \longmapsto     & \tau^\natural(f) = \tau f \natural
\end{array}
\end{align*}
Note also that this remains valid if $\tau$ is a trace on an $\L$-bimodule $\mathscr{M}$ (i.e a linear map that vanishes on $[\L,\mathscr{M}]$). 
\end{itemize}

\subsection{The JLO cocycle of an unbounded K-cycle}

The remarkable feature of Quillen's formalism of algebra cochains is that it recovers most of the important (periodic) cyclic cohomology classes via an abstraction of classical methods and constructions in Chern-Weil theory and Chern characters. We recall in this section how to construct the JLO cocycle via algebra cochains. \\

Let $H$ be a $\Z_2$-graded Hilbert space, and let $\L=\L(H)$ be the algebra of bounded operators on it. Suppose that $\A$ acts on $H$ as (even) bounded operators on $H$, and that the `heat operator' $e^{-tD^2}$ is a trace-class operator on $H$, for any $t>0$. One works with improper cochains in
\[ \Hom(\bB, \L) := \prod_{n \geq 0} \, \Hom(\bB_n, \L)\] 
substituting the direct sum in the original definition of the left-hand-side given in previous section by the direct product. Rigorously, one needs to work within the framework of \emph{entire cyclic cohomology}, and consider cochains subject to a certain growth condition. However, as stated in the introduction, we are in this article interested only in the formal aspects of the theory. Furthermore, we won't treat the case of an odd K-cycle. \\

%$(H, \rho, D)$, and let $\L = \L(H)$ denote the algebra of bounded operators on the $\Z_2$-graded Hilbert space $H$. Rigorously, to place ourselves within this framework, $\A$ and $\bB(\A)$ should be equipped with an appropriate bornology to include \emph{entire cochains}, which are elements of
%\[ \widehat{\Hom}(\bB(A), \L) \subseteq \prod_{n \geq 0} \, \Hom(\bB_n(\A), \L)\] subject to a certain growth condition (above, $\bB(\A) = \A^{\otimes n}$)). However, as stated in the introduction, we are in this article interested only in the formal aspects of the theory and will mostly stray away from these analytic matters. \\

Now, the main point is to view $D$ as a superconnection form associated to the superconnection 
\[ \D = \delta + \rho + D \]
Because $\rho$ is an algebra homomorphism $\rho: \A \to \L$, the curvature $\D^2$ is: 
\[ \D^2 = (\delta + \rho + D)^2 = D^2 + [D, \rho]. \] 
Using the classical Duhamel perturbation series: for any $t >0$, 
\[ e^{-tD^2} = \sum_{n \geq 0} (-t)^n \int_{\Delta_n} e^{-s_0 t D^2} [D, \rho] e^{-s_1 t D^2} \ldots [D, \rho] e^{- s_n t D^2} ds_1 \, \ldots \, ds_n \]
where $\Delta_n = \{(s_0, s_1, \ldots, s_n) \in \R^{n+1} \, ; \, s_i \geq 0, \sum_i s_i = 1 \}$ denotes the standard $n$-simplex. \\

Note also that by the $\theta$-summability conditions, $e^{-tD^2}$ and $[D,\rho]$ do belong to $\L$, but this is not the case for $D$. Consequently, $e^{-t\D^2}$ is in the (extended) cochain algebra, whereas $\D$ is not. \\

Finally, let $\Tr_s$ be the supertrace defined on the ideal of trace-class operators within $\L$, and set: 
\[ \psi = \Tr_s^{\natural} \big(\partial\rho \cdot e^{-t\D^2} \big) \in \Hom(\Omega \A, \C) \quad ; \quad \varphi = \Tr_s(e^{-t\D^2}) \in  \Hom(\bB, \C)\]
In particular,
\begin{gather*}
\psi_{n+1}(a_0, \ldots, a_n) = (-t)^n \int_{\Delta_n} \Tr_s \big(a_0 e^{-s_0 t D^2} [D, a_1] e^{-s_1 t D^2} \ldots [D, a_n] e^{-s_n t D^2} \big) \, ds_1 \, \ldots \, ds_n\\
\varphi_n(a_1, \ldots, a_n) = \psi_{n+1}(1, a_1, \ldots, a_n)
\end{gather*}
and we see that $\psi$ is the JLO cocycle. Quillen proves that $\psi$ is a $(b,B)$-cocycle by applying Proposition \ref{thm: Quillen (b,B)-cocycles} to the pair $(\psi, \varphi)$. 

\begin{theorem} \label{thm: Quillen (b,B)-cocycles bis} We have the following relations:
\[ \delta\varphi = \beta (\pm \psi) \quad ; \quad \delta(\pm \psi) = \ol{\partial}\varphi \]
where $\pm$ depends on the parity of $\psi$. 
\end{theorem}

\begin{proof} Since we are only interested in the second relation, we will not deal with the other. We follow Quillen's exposition of \cite[Theorems 7 and 8]{Qui88} almost verbatim. \\

One first establishes the following Bianchi identity
\[ [\D, e^{-\D^2}] = \delta(e^{-\D^2})+ [\rho + D, e^{-\D^2}] = 0 \] 
which is a result of the differentiation formula:
\begin{equation} \label{fml: differentiation formula for derivations}\alpha(e^{-\D^2}) = \int_0^1 e^{-s\D^2} \alpha(-\D^2) e^{-(1-s)\D^2} \, ds  \end{equation}  
for any derivation $\alpha$ (chosen in this case to be $\alpha = [\D, \centerdot ] = \ad \, \D$), together with the fact that $\alpha(\D^2) = [\D, \D^2] = \delta(e^{-\D^2})+ [\rho + D, e^{-\D^2}] = 0$ (this can be verified by expanding the commutator, or simply by observing that $\D^2$ has even degree). \\

The next step is to evaluate $[\D, \psi] = [\D, \tau^{\natural} \big( \partial \rho \cdot e^{-D^2}\big)]$, which is constituted of the three terms in the middle column below:  
\begin{equation*}
\begin{array}{r c c l}
& \delta \tau^\natural \big(\partial \rho \cdot e^{-\D^2}\big) & = & \tau^\natural \big(\partial(-\rho^2) e^{-\D^2} - \partial \rho \cdot \delta e^{-\D^2}\big) \\
0 = & \tau^\natural\big([\rho, \partial \rho \cdot e^{-\D^2}]\big) & = &\tau^\natural \big((\rho \cdot \partial \rho + \partial \rho \cdot \rho)e^{-\D^2} - \partial \rho \cdot [\rho, e^{-\D^2}] \big) \\
0 = & \tau^\natural\big([\nabla, \partial \rho \cdot e^{-\D^2}]\big) & = &\tau^\natural\big([\nabla, \partial\rho] e^{-\D^2} - \partial \rho \cdot [\nabla, e^{-\D^2}]\big),
\end{array}
\end{equation*}   
the first line uses the fact that $\rho$ satisfies $\delta \rho + \rho^2 = 0$, as an algebra homomorphism. Adding these three equations, the rightmost terms of the third column are killed because of the first Bianchi identity, and among the leftmost terms of that same column, only $[\nabla, \partial \rho] e^{-\D^2} = \partial[\nabla, \rho] e^{-\D^2}$ survives. Therefore, 
\begin{equation*} 
\delta \tau^{\natural}\big(\partial \rho \cdot e^{-\D^2}\big) = \tau^{\natural}\big(\partial \D^2 \cdot e^{-\D^2}\big)
\end{equation*}
In addition, 
\begin{equation*} 
\overline {\partial}\tau\big(e^{-\D^2}\big) = \tau^\natural\big(\partial e^{-\D^2}\big) = \int_0^1 \tau^\natural\big(e^{-t\D^2} \cdot \partial \D^2 \cdot e^{-(1-t)\D^2}\big) dt = \tau^\natural\big(\partial \D^2 \cdot e^{-\D^2}\big),
\end{equation*}
the last equality comes from the trace property. As we have: \begin{equation*} 
\delta \tau^{\natural}\big(\partial \rho \cdot e^{-\D^2}\big) = \overline {\partial}\tau\big(e^{-\D^2}\big)\end{equation*}
which concludes the proof.
\end{proof}

\subsection{The JLO cocycle of an unbounded KK-cycle} \label{sec: bivariant JLO} 

Let $A, B$ be $C^*$-algebras. Given an unbounded $A$-$B$ Kasparov bimodule $(\E, \rho, D)$, the main point is, as in previous subsection, to view $D$ as a kind of superconnection form, in the context where $\L$ would basically be replaced by (a differential form version of) the algebra $\L_B(\E)$ of continuous $B$-linear maps. It is possible to extend the notion of $\theta$-summability to this bivariant context (see \cite{Per2002} for a detailed account of this), but we won't need the full extent of the theory for the geometric applications we have in mind, so it will be sufficient to remain at a formal level. \\   

Let $\A$ and $\B$ be smooth\footnote{i.e of Fr\'{e}chet-type and dense, but again, topology won't play any role in this article.} $*$-subalgebras of $A$ and $B$. By a slight abuse of language, we will consider the Kasparov bimodule $(\E, \rho, D)$ relatively to $\A$ and $\B$ instead of the associated algebras. Assume that the `heat operator' $e^{-tD^2}$ is densely-defined and extends to a bounded endomorphism of $\E$. For us, it is enough to limit ourselves to assuming $\E$ of the form $\E=H \otimes \B$, where $H$ is a $\Z_2$-graded Hilbert space. Extend $\E$ to a $\A$-$\Omega\B$-bimodule of `$\E$-valued differential forms':
\[ \Omega\E = \E \otimes_{\B} \Omega\B = H \otimes \Omega\B \] 
The differential $d$ on $\Omega \B$ naturally induces a differential on $\Omega\E$ that we continue to denote $d$:
\[ d(h \otimes \omega) = (-1)^{| h |} h \otimes d\omega \quad ; \quad \forall \xi \otimes \omega \in \Omega \E = H \otimes \Omega\B, \]  
This turns the algebra of endomorphisms $\L=\L_{\Omega \B}(\Omega \E)$ into a dg-algebra via the differential: 
\[ d\varphi = d \circ \varphi + (-1)^{| \varphi |} \varphi \circ d \quad ; \quad \forall \varphi \in \L. \]
The representation $\rho: \A \to \L_{\B}(\E)$ extends naturally to an algebra homomorphism map $\rho: \A \to \L$. Then, let us consider the `superconnection'
\[ \D = \delta + d + \rho + D, \]
viewing $D$ as a superconnection form. Since $\rho$ is an algebra homomorphism, we have $\delta \rho + \rho^2 = 0$ and viewing $d$ and $D$ as $0$-cochains, one sees that the `curvature' of $\D$ is 
\[ \D^2 = d(\rho+D) + [D, \rho] + D^2. \]
A different way to write this is to consider the superconnection $\nabla=d+D$ with curvature $\nabla^2=dD+D^2$; we then have
\[ \D^2 = \nabla^2 + [\nabla, \rho]. \]

Suppose we are given a trace $\tau: \mathcal{I} \to \Omega \B$ is a trace on an $\L$-bimodule $\mathcal{I}$, and assume that the heat operator $e^{-t\D^2}$ is also in $\mathcal{I}$ for every $t>0$. Then we can simply imitate the previous algebra cochains constructions without modification, i.e by introducing the Hochschild and bar cochains %and consider its extension $\tau \otimes 1$ to $\L_{\Omega\B}(\Omega\E)$, that we denote again by $\tau$. Then, define the following Hochschild and bar cochains with values in $\Omega\B$: 
\begin{gather*}
\psi = \tau^{\natural}\big(\partial \rho \cdot e^{-\D^2}\big) \in \Hom(\Omega \A, \Omega \B), \\
\varphi = \tau\big(e^{-\D^2}\big) = \psi(1, \centerdot ) \in \Hom(\Omega_1 \bB(\A), \Omega \B),
\end{gather*} 
and prove that $\psi \in \Hom(\Omega \A, \Omega \B_{\natural})$ is a cocycle from the $(b,B)$-complex of $\A$ to the de Rham--Karoubi complex of $\B$, i.e $\psi (B-b) = \pm \natural d \psi$. Observe that this construction extends straightforwardly if we replace the differential $d$ by a right connection $\nabla'$ on the $\B$-module $\E$, i.e a linear map $\nabla^{\E}: \Omega\E \to \Omega\E$ of degree $1$ such that 
\[ \nabla^{\E}(\xi \otimes \omega) = \nabla^{\E}(\xi)\omega + (-1)^{|\xi|} \xi \otimes d\omega \quad ; \quad \forall \xi \otimes \omega \in \Omega\E = \E \otimes_{\B} \Omega \B. \]

When the algebra $\B$ is commutative, then $\psi$ may really be considered as a representative of the bivariant Chern character of the $(\E, \rho, D)$. In general, this is not exactly the case and the cochain described above constitutes only a part of it. See \cite{Per02} for further details.  %the Appendix for the full construction of the Chern character. 

\section{Bivariant JLO cocycle and the Mathai-Quillen form}

We come now to the main subject of this article, namely the verification that Quillen's algebra cochains machinery is consistent with his superconnection formalism, which is kind of apparent when studying the history of the subject via his mathematical journals, but has surprisingly never really been brought to light, despite many potentially interesting applications.  

\subsection{The Bott element}
We consider first the Bott generator of the K-theory of $\R^{2n}$, represented by the Kasparov bimodule
\[ [\text{Bott}] = \big[(\E=S \otimes C_0(\R^{2n}), \rho, L) \big] \in \KK(\C, C_0(\R^{2n})). \]
where $S$ is the space of $2n$-dimensional spinors (hence $\E$ is the space of sections of the trivial bundle $E=S \times \R^{2n}$), $\rho:\C \to \L_{C_0(\R^{2n})}(\E)$ is the obvious homomorphism sending the unit $e \in \C$ to the identity. The operator $L(x)$ is the section of $\End(E)$ defined as the Clifford multiplication by $x \in \R^{2n}$ on the fiber $E_x = S$. \\

The space $S$ is a $\Z_2$-graded vector space of dimension $2^n$ equipped with inner product and anti-commuting hermitian involutions $\gamma^1, \ldots, \gamma^{2n}$ of odd degree that  generate the Clifford algebra of $\R^{2n}$. Identifying the latter with $\End(S)$, the operator $L$ writes (using the Einstein convention):
\[ L(x) = \sqrt{t} x_{\mu} \gamma^{\mu} . \]
The chirality element $\Gamma = (-\i)^n \gamma^1 \ldots \gamma^{2n}$ satisfies $\Gamma^2=1$ and is then a grading operator, which yields the supertrace is $\tr_s=\tr(\Gamma \, \centerdot)$. The latter satisfies the following identities, 
\begin{align*}
&\tr_s(\gamma^1 \ldots \gamma^{2n}) = (2\i)^n & &;& & \tr_s(\gamma^{i_1} \ldots \gamma^{i_p}) = 0 \text{ for } p < 2n. &&&&&&&&&&&
\end{align*}  
Finally, consider the trivial connection $\nabla^{\E}=d$ on $E$ induced by the de Rham differential, and for every $t>0$, form the superconnection $\nabla = d+\sqrt{t}L$ on $E$, which has curvature \[\nabla^2 = t L^2 + [d,tL] = t \Vert x \Vert^2 + \sqrt{t} dx_{\mu}\gamma^{\mu} .\] Then, the algebra cochain $\psi$ is given by
\[ \psi = \tr_s^{\natural}\big(\partial \rho \cdot e^{-\D^2}\big) \in \Hom(\Omega \C, \Omega(\R^{2n})) \]
where $\D = \delta + \rho + d + L$ is the `superconnection' acting at the level of cochains, with `curvature' 
\[\D^2 = \nabla^2 + [\nabla, \rho].\] 
Since the image of $\rho$ consists of multiples of the identity, the commutators $[\nabla, \rho]$ vanishes, so the only non-zero term in the Duhamel expansion of $e^{-\D^2}$ yields only one term $e^{-\nabla^2}$. Hence, $\psi$ reduces to a 1-cochain $\psi_1$, 
\begin{align*}
\psi_1 &= \tr_s(e^{-\nabla^2}) \\
	   &= e^{-t \Vert x \Vert^2} \tr_s\big((1-\sqrt{t}dx_{1} \gamma^{1}) \ldots (1-\sqrt{t}dx_{2n} \gamma^{2n})\big) \\
	   &= e^{-t \Vert x \Vert^2} (-\sqrt{t})^{2n} \tr_s(dx_{1} \gamma^1  \ldots dx_{2n} \gamma^{2n})\\
	   &= (2it)^n e^{-t \Vert x \Vert^2} dx_1 \ldots dx_{2n}    
\end{align*}
where the passage from the third line to the fourth uses the fact that $dx^{\mu}$ and $\gamma^{\mu}$ anti-commute in the algebra $\End(\Omega \E) = \End(S) \otimes \Omega(\R^{2n})$. \\

Consider now the Dirac element:
\[ [\slashed{D}] = \big[(L^2(\R^{2n}, S), \pi, D)\big] \in \KK(C_0(\R^{2n}), \C). \]
In this context, Bott periodicity is the fact that the Kasparov product $\big[\text{Bott} \big] \otimes [\slashed{D}] = 1$. On the other hand, via the JLO formula, the Chern character of $[\slashed{D}]$ is represented by the $(b,B)$-cocycle
\[ \ch[\slashed{D}](a_0 da_1 \ldots da_{2n}) = \dfrac{1}{(2\pi\i)^n}\int_{\R^{2n}} \, a_0 da_1 \ldots da_{2n} \in \Hom(\Omega C_c^{\infty}(\R^{2n}), \C)\]
i.e the fundamental class of $\R^{2n}$. Then the pairing $\ch[\slashed{D}] \circ \psi = 1$, which is consistent with Bott periodicity. 

\subsection{The Mathai--Quillen Thom form}

Let $E \overset{\pi}{\rightarrow} X$ be a complex vector bundle of rank $m$ over a smooth manifold $X$ of rank $m$, equipped with compatible Hermitian inner product and connection. Let $\Lambda E^*$ be the exterior power of its dual bundle $E^*$ endowed with the usual $\Z_2$-grading via forms of even/odd degree. We consider the \emph{Thom element}:
\[ [\mathscr{T}_E] = [(\E=C_0(E, \pi^*\Lambda E^*), \rho, L)] \in \KK(C_0(X), C_0(E)), \]  
where $\rho$ is the natural action of $C_0(X)$ on $\E$ through multiplication, and $L$ is the odd degree endomorphism on $\pi^{*}\Lambda E^*$ defined as follows:
\[ L(\xi) = \i\big((\xi^* \wedge \centerdot) \, - \, \iota_{\xi}\big) \, , \]
i.e that for every $\xi \in E$, $L_{\xi}$ acts on the fibers $(\pi^{*}\Lambda E^*)_{\xi} = \Lambda E^*_{\pi(\xi)}$. In the formula above, $\xi^* \in E^*_{\pi(\xi)}$ is the linear functional on $E_{\pi(\xi)}$ given by inner product with $\xi$, and $\iota_{\xi}$ denotes interior product with $\xi$ seen as a linear functional on $E^*_{\pi(\xi)}$. Note that $L_{\xi}^2=-\Vert \xi \Vert^2$; in other words, if we see $\pi^*\Lambda E^*$ as a Clifford module bundle over $E$, the endomorphism $L$ is the fiberwise Clifford multiplication operator.  \\  

The connection in $E$ induces connections in $\Lambda E^*$ and $\pi^*(\Lambda E^*)$, and therefore a superconnection $\nabla + L$ on $\pi^*(\Lambda E^*)$. Hence, setting\footnote{or eventually, replacing functions with compact support with ones in the Schwartz class, but this would really enforce the use of entire cyclic theory and bornologies} $\A=C_c^{\infty}(X)$ and $\B=C_c^{\infty}(E)$ we can form the algebra cochain superconnection $\D = \delta + \rho + \nabla + L$, whose curvature is
\[ \D^2 = (\nabla + L)^2 + [\nabla, \rho]. \]
Using a Duhamel expansion, the cochain $\psi = \tr_s^{\natural}\big(\partial \rho \cdot e^{-\D^2}\big) \in \Hom(\Omega C_c^{\infty}(X), \Omega C_c^{\infty}(E)_{\natural})$ writes:
\setlength{\mathindent}{15mm}
\begin{align*}
\psi(a_0, \ldots, a_k)=\sum_{k \geq 0} \int_{\Delta_k} \tr_s\left( \rho(a_0) e^{-s_0 (\nabla+L)^2} [\nabla, \rho(a_1)] e^{-s_1 (\nabla+L)^2} \ldots [\nabla, \rho(a_k)] e^{- s_k (\nabla+L)^2}\right) ds_1 \, \ldots \, ds_k
\end{align*}
where $\Delta_k$ is the standard $k$-simplex. \\

\setlength{\mathindent}{15mm}
On the one hand, $[\nabla, \rho(a_i)] = L(da_i)$. On the other hand, the work of Mathai--Quillen \cite{MQ86} provides an exact formula for $\tr_s(e^{(\nabla+L)^2})$:
\[ \tr_s(e^{(\nabla+L)^2}) = \left(\dfrac{\i}{2\pi}\right)^{-m} \text{det}\left(\dfrac{1-e^{\Omega}}{\Omega}\right) U \]
where $U$ is the Gaussian-shaped Thom form on $E$, whose integral over every fibre of $E$ is equal to $1$, and $\Omega$ is the curvature of a connection on the manifold $E$, so the factor involving the determinant represents the Todd genus of $E$. In the end, we find
\[ \psi(a_0, \ldots, a_k)= \left(\dfrac{\i}{2\pi}\right)^{-m} a_0 da_1 \, \ldots \, da_k \wedge \text{det}\left(\dfrac{1-e^{\Omega}}{\Omega}\right) U. \] 

Let us compare this with Kasparov's formulation of the Atiyah-Singer index theorem. Let $M$ be a closed manifold; choosing an almost complex structure on $TM$ (we will not really make a distinction between a bundle and its dual from that point), the Dolbeault operator $\ol{\partial}$ induces a K-homology fundamental class $[\ol\partial_{TM}] \in \KK(C_0(TM), \C)$. The latter is universal in the sense that the index of any elliptic operator $P$ with symbol class $[\sigma_P] \in \KK(\C, C_0(TM))$ is equal to the Kasparov product $[\sigma_P] \otimes [\ol\partial_{TM}]$. \\ %via the elliptic operator $\ol{\partial} + \ol{\partial}^*$,\\

Furthermore, consider an embedding $M \hookrightarrow \R^n$ with normal bundle $N \to M$, which induces an embedding $TM \hookrightarrow T\R^n$ with normal bundle $E=TN \to X=TM$. Note that $TN$ is isomorphic to the pull-back of $N \oplus N$ to $TM$, so it may naturally be equipped with a complex structure. The topological index is the composition 
\[ \index_t: K^0(TM) \overset{\text{Thom}}{\longrightarrow} K^0(TN) \overset{\text{excision}}{\longrightarrow} K^0(T\R^n) \overset{\text{Bott}}{\longrightarrow} \Z, \]
The first and last maps are respectively the Thom isomorphism and Bott periodicity, which may be described via the right Kasparov product with $[\mathscr{T}_{TN}] \in \KK(C_0(TM), C_0(TN))$ for the former, and the right Kasparov product with the class $[\slashed{D}] \in \KK(C_0(\R^{2n}), \C)$ for the latter. As for the excision map, it is induced by the inclusion $j: C_0(TN) \hookrightarrow C_0(T\R^n)$, where $TN$ is identified with an (open) tubular neighborhood $W$ of $TM$. \\

Kasparov proves the index theorem by deriving the following KK-factorization (this is not completely obvious): 
\[ [\ol{\partial}_{TM}] = [\mathscr{T}_{TN}] \otimes j^*[\slashed{D}_{T\R^n}]. \] 

Back to cyclic theory, recall that the JLO cocycle associated to $[\slashed{D}]$ is:
\[ \ch[\slashed{D}](a_0 da_1 \ldots da_{2n}) = \dfrac{1}{(2\pi\i)^n}\int_{\R^{2n}} \, a_0 da_1 \ldots da_{2n} \in \Hom(\Omega C_c^{\infty}(\R^{2n}), \C).\]
If we combine it with the cochain $\psi$ related to the Mathai-Quillen form, the multiplicativity of the Todd class yields, (picking up entries supported inside the tubular neighborhood $W$):
\[ (\ch[\slashed{D}] \circ \psi)(a_0 da_1 \ldots da_k) = \left(\dfrac{\i}{2\pi}\right)^{2\dim(M)} \int_{TM} \text{Todd}(TM\otimes\C) \wedge a_0 da_1 \, \ldots \, da_k  \]
where $\text{Todd}(TM\otimes\C)$ is the pull-back of the Todd class of the bundle $TM \to M$ to $TM$. This is consistent with the index theorem (and therefore the KK-factorization above).

\providecommand{\bysame}{\leavevmode\hbox to3em{\hrulefill}\thinspace}
\providecommand{\MR}{\relax\ifhmode\unskip\space\fi MR }
\providecommand{\MRhref}[2]{%
  \href{http://www.ams.org/mathscinet-getitem?mr=#1}{#2}
}
\providecommand{\href}[2]{#2}

\end{document}